\documentclass[11pt]{article}

\title{A deterministic proof of Loewner energy reversibility \\via local reversals}
\author{Jinwoo Sung\thanks{University of Chicago. \protect\url{jsung@math.uchicago.edu}}}
\date{}

\usepackage[margin=1.5in]{geometry}
\usepackage{amsthm,amsmath,amsfonts,amssymb}
\usepackage{graphicx}
\usepackage{enumerate,enumitem}
\usepackage{cite,url}
\usepackage[dvipsnames]{xcolor}

\usepackage{hyperref}
\hypersetup{
    colorlinks=true, 
    linkcolor=black, 
    urlcolor=black, 
    citecolor=black,
    linktoc=all 
}

\usepackage[font=normal,labelfont=bf,labelsep=quad]{caption}

\allowdisplaybreaks

\setlist[enumerate]{topsep = 1ex, leftmargin=1cm, itemsep= -2pt}
\setlist[itemize]{topsep = 1ex, leftmargin=1cm, itemsep= -2pt}

\let\OLDthebibliography\thebibliography
\renewcommand\thebibliography[1]{
  \OLDthebibliography{#1}
  \setlength{\parskip}{1pt}
  \setlength{\itemsep}{2pt}
}

\newtheorem{thm}{Theorem}[section]

\newtheorem{lem}[thm]{Lemma}
\newtheorem{prop}[thm]{Proposition}

\newtheorem{thmx}{Theorem}

\theoremstyle{definition} 
\newtheorem{df}[thm]{Definition}

\newtheorem{remark}[thm]{Remark}

\numberwithin{equation}{section}

\newcommand{\mc}[1]{\mathcal{#1}}
\newcommand{\mb}[1]{\mathbb{#1}}

\renewcommand{\Im}{\mathrm{Im}\,}

\newcommand{\R}{\mathbb{R}}
\newcommand{\C}{\mathbb{C}}
\newcommand{\Chat}{\hat{\mathbb{C}}}
\newcommand{\uhp}{\mathbb{H}}
\newcommand{\ud}{\mathbb{D}}
\newcommand{\ep}{\varepsilon}

\newcommand{\hcap}{\mathrm{hcap}}
\newcommand{\SLE}{\mathrm{SLE}}
\newcommand{\rev}{\mathrm{R}}

\newcommand{\boundary}{\partial}

\begin{document}

\maketitle

\begin{abstract}
    We give a new proof of the orientation reversibility of chordal Loewner energy by reversing the orientation of a chord in partial increments. This fact was first proved by Yilin Wang \cite{W1} using the reversibility of chordal Schramm--Loewner evolution (SLE) along with the interpretation of Loewner energy as the large deviation rate function of chordal SLE$_\kappa$ as $\kappa \to 0$.  Our method is similar in spirit to Dapeng Zhan's proof \cite{Zhan} of chordal SLE$_\kappa$ reversibility for $\kappa \in (0,4]$, though it is purely deterministic. As a key step in our proof, we establish that a minimimal energy chord among those passing through a fixed finite set of points is a piecewise hyperbolic geodesic.
\end{abstract}

\section{Introduction}

The Loewner energy of a chord is defined as the Dirichlet energy of the driving function of the corresponding Loewner chain. It was introduced in independent works by Friz and Shekar \cite{FrizShekhar15} and Yilin Wang \cite{W1}, the latter of whom considered it in the context of large deviations for a family of random curves called the chordal Schramm--Loewner evolution\footnote{The $\kappa \to 0$ limit of chordal SLE$_\kappa$ was also considered in \cite[Section~9.3]{Dub_comm}.} ($\SLE_\kappa$) as $\kappa \to 0$. Loewner energy has since been identified in terms of various geometric and probabilistic quantities, hinting at a deeper connection between these fields. We refer the reader to the survey articles \cite{sle_ld_survey,W_AMS} for various perspectives on Loewner energy.

Let $\gamma$ be a simple curve from 0 to $\infty$ in the upper half-plane $\uhp = \{z\in \C: \Im z>0\}$. We parameterize $\gamma$ by $\R_+ = (0,\infty)$ so that, if $g_t:\uhp\setminus \gamma((0,t]) \to \uhp$ is the unique conformal map with the normalization $g_t(z) - z \to 0$ as $z\to \infty$, then the expansion $g_t(z) = z + 2a_t z^{-1} + O(z^{-2})$ as $z\to\infty$ satisfies $a_t= t$ for every $t>0$. By Carath\'eodory's theorem, we can extend $g_t$ to the prime ends of $\R \cup \gamma((0,t])$. Let $\lambda_t := g_t(\gamma(t))$, which we call the \textit{driving function} of $\gamma$. The \textit{Loewner energy} of $\gamma$ is defined as
\begin{equation}\label{eq:uhp-def}
    I_{\uhp;0,\infty}(\gamma) := \frac{1}{2}\int_0^\infty \left|\frac{d\lambda_t}{dt}\right|^2 dt
\end{equation}
if $t\mapsto \lambda_t$ is absolutely continuous and the energy is set to be infinite otherwise.

Loewner energy is defined to be a conformally invariant quantity, which is natural given the conformal invariance of chordal SLE$_\kappa$.
That is, if $D\subsetneq \C$ is a simply connected domain and $\gamma$ is a crosscut (henceforth called a \textit{chord}) from a prime end $a$ to another prime end $b$, we define the \textit{Loewner energy} of $\gamma$ as
\begin{equation}\label{eq:energy-conformal-invariant}
    I_{D,a,b}(\gamma) := I_{\uhp,0,\infty}(\varphi(\gamma))
\end{equation}
where $I_{\uhp;0,\infty}$ is defined as in \eqref{eq:uhp-def} and $\varphi:D\to\uhp$ is a conformal transformation with $\varphi(a) = 0$ and $\varphi(b) = \infty$. It is straightforward to check that this definition does not depend on the choice of such map $\varphi$.

A fundamental property of Loewner energy, called \textit{reversibility}, is that the Loewner energy of a chord does not depend on the choice of its orientation. It is not obvious a priori that the two quantities are related due to the directionality in the definition of Loewner energy.
If $\gamma$ is a chord in $\uhp$ going from 0 to $\infty$, let $\gamma^\rev$ denote the same chord in the reverse orientation, going from $\infty$ to $0$ in $\uhp$. If we denote $\iota: z\mapsto -1/z$, we see from \eqref{eq:energy-conformal-invariant} that $I_{\uhp,\infty,0}(\gamma^\rev) = I_{\uhp,0,\infty}(\iota(\gamma^\rev)) = \frac{1}{2}\int_0^\infty |\dot \lambda_t^\rev|^2\,dt$ where $\lambda_t^\rev$ is the driving function corresponding to the curve $\iota(\gamma^\rev)$. The map $(\lambda_t)_{t\in [0,\infty)} \mapsto (\lambda_t^\rev)_{t\in [0,\infty)}$ induced by reversing of the chord is rather intricate and there does not exist an explicit formula except for extremely simple examples. Nevertheless, Wang showed in \cite{W1} that the Loewner energy of $\gamma^\rev$ is equal to that of $\gamma$. The purpose of this article is to give another proof of this fact.

\begin{thm}[Reversibility]
    \label{thm:main}
    Let $D$ be a simply connected domain and let $a$ and $b$ be any two distinct prime ends of $D$. For a chord $\gamma$ in $D$ from $a$ to $b$, let $\gamma^\rev$ be the same chord with its orientation reversed to trace from $b$ to $a$. Then, $I_{D;a,b}(\gamma) = I_{D;b,a}(\gamma^\rev)$.
\end{thm}

Let us briefly describe previous proofs of Loewner energy reversibility and explain how ours differs from them. The original proof in \cite{W1} relied upon the reversibility of chordal SLE$_\kappa$ for $\kappa \in (0,4]$, which was first established in \cite{Zhan} and proved subsequently with different methods in \cite{IG2,lawler_yearwoood_sle_reversiblity}. By interpreting Loewner energy $I_{\uhp;0,\infty}$ as the large deviation rate function as $\kappa \to 0$ for chordal SLE$_\kappa$ in $\uhp$ from 0 to $\infty$, Wang showed that a finite-energy chord must have the same Loewner energy in the reverse orientation.

Wang gave a deterministic proof of Theorem~\ref{thm:main} in \cite{W2} through an identification between $I_{\uhp,0,\infty}(\gamma)$ and the Dirichlet energy of $\log |f'|$, where $f$ is the uniformizing map on each component of $\uhp \setminus \gamma$. The reversibility of Loewner energy follows immediately since $f$ does not depend on the orientation of the chord $\gamma$. Based on this result, she revealed in the same work an unexpected link between Loewner chains, conformal welding, and the K\"ahler geometry of the Weil--Petersson Teichm\"uller space.

The key difference in our proof of Loewner energy reversibility from the previous deterministic one is that, rather than ``lifting" $\gamma$ to the space of chords without orientation, we consider a ``continuous" reversal of its orientation. In particular, our proof is based on the \textit{commutation relation} for Loewner energy (Theorem~\ref{thm:commutation}), which says that the order in which we calculate the total energy of two non-intersecting Jordan arcs does not change this quantity. Despite being closely related with the commutation relation for chordal SLE \cite{Dub_comm}, this identity was obtained deterministically in \cite{W1} from the definition \eqref{eq:uhp-def} of Loewner energy through a careful calculation identifying how the driving function of one curve changes when the other curve is grown first. 

\subsection{Proof overview}
Here is an overview of our proof of Theorem~\ref{thm:main}. Given a chord $\gamma$ in a simply connected domain $D$ from $a$ to $b$, we approximate it with chords which pass through finite subsets of its trace and satisfy energy reversibility. In particular, given a finite subset $\underline z \subset \gamma$, we choose a chord $\gamma_{\underline z}$ with minimal energy among those which pass through $\underline z$ between $a$ and $b$. Using an inequality version of the commutation relation for simple curves terminating at the same point (Lemma~\ref{lem:comm-ineq}), we identify that $\gamma_{\underline z}$ is \textit{piecewise geodesic relative to $\underline z$}. This means that $\underline z$ partitions $\gamma$ into simple curves $\gamma_1,\dots,\gamma_n$ such that each $\gamma_k$ is a hyperbolic geodesic in the domain $D \setminus (\bigcup_{j\neq k} \gamma_j)$. The commutation relation further implies that the Loewner energy of any piecewise geodesic chord is invariant under reversing its orientation (Proposition~\ref{prop:minimizer}). As $\underline z$ becomes dense in $\gamma$, the Loewner energy of $\gamma_{\underline z}$ increases to that of $\gamma$; this holds in both directions, thus proving that the Loewner energies of $\gamma$ and its orientation reversal are the same.

Our proof technique can be applied to show that when $\underline z$ contains $n\geq 2$ points, a Loewner energy minimizer within each isotopy class of chords relative to $\underline z$ is piecewise geodesic. There are infinitely many such isotopy classes, as can be seen via an identification with the elements of the braid group on $n$ strands (see Remark~\ref{rem:braid}).

\begin{thm}\label{prop:isotopy}
    Given a chord $\gamma$ in $D$ between prime ends $a,b$ passing through a finite set of points $\underline z \subset D$, let $\mc X(D;a,b,\underline z,\gamma)$ be the set of chords in $D$ between $a$ and $b$ that can be obtained through an isotopy of $\gamma$ leaving $\underline z\cup \{a,b\}$ pointwise fixed. Then, given any chord $\gamma$ and a finite subset $\underline z \subset \gamma$, there exists a chord $\gamma^{\mathrm{min}} \in \mathcal X(D;a,b;\underline{z},\gamma)$ such that 
    \begin{equation}\label{eq:isotopy-min}
        I_{D;a,b}(\gamma^{\mathrm{min}}) = \inf_{\mc X(D;a,b;\underline{z},\gamma)} I_{D;a,b}(\gamma) < \infty.
    \end{equation}
    Moreover, if $\gamma^{\mathrm{min}}$ is any chord in $\mathcal X(D;a,b;\underline{z},\gamma)$ satisfying \eqref{eq:isotopy-min}, then it is piecewise geodesic relative to $\underline z$.
\end{thm}

In the case that $\underline{z}$ contains a single point $z_1$, the piecewise geodesic property as well as the uniqueness of the energy minimizing chord passing thorough this point was verified by an explicit identification of this chord in \cite[Prop.\ 3.1]{W1} via SLE$_{0+}$ large deviations and in \cite[Thm.\ 3.3(i)]{energy-minimizer} via deterministic Loewner chain calculations. 
In fact, it was shown in \cite[Thm.\ 3.9]{piecewise-geodesic-i} that this is the unique $C^1$ chord which is piecewise geodesic relative to $z_1$. Based on an analogous result for the case of loops \cite{piecewise-geodesic-ii}, we conjecture that when $\underline z$ contains more than one point, there exists a unique Loewner energy minimizer within each isotopy class $\mc X(D;a,b;\underline z, \gamma)$ and it is the unique $C^1$ chord in this isotopy class which is piecewise geodesic relative to $\underline z$.

\begin{remark}
Our proof of Loewner energy reversibility is closely related to the proof of chordal SLE$_\kappa$ reversibility for $\kappa \in (0,4]$ in \cite{Zhan}. Zhan establishes this fact by constructing a coupling of an SLE$_\kappa$ chord $\gamma_1$ going from $x_1$ to $x_2$ in the domain $\uhp$ for fixed $x_1,x_2\in \R$ with another SLE$_\kappa$ chord $\gamma_2$ in $\uhp$ between the same endpoints in reverse direction so that both traverse the same random set of points $\underline{z}\subset \uhp$ visited by $\gamma_1$ at stopping times $t_1 < \dots< t_n$. In fact, the reverse chord $\gamma_2$ is constructed by replacing each curve $\gamma_1([t_{k-1},t_k])$, as $k$ decreases from $n+1$ to 1, with a reweighted version of chordal SLE$_\kappa$ going from $\gamma_1(t_k)$ to $\gamma_1(t_{k-1})$ in the complement of $\gamma_1([0,t_{k-1}])$ and the already constructed part of $\gamma_2$ (cf.\ proof of Proposition~\ref{prop:minimizer}).
The random set $\underline{z}$ is taken to become denser and denser on $\gamma_1$, ultimately converging to a coupling where the two SLE$_\kappa$ chords are almost surely orientation reversals of each other. 

The relation to our proof of Loewner energy reversibility is that, conditioned on the points $\underline{z}$, the law of the chord $\gamma_1$ (resp.\ $\gamma_2$) in the above coupling is that of a chordal SLE$_\kappa$ curve in $\uhp$ from $x_1$ to $x_2$ (resp.\ $x_2$ to $x_1$) conditioned to pass through $\underline{z}$. Since $I_{\uhp;x_1,x_2}$ (resp.\ $I_{\uhp;x_2,x_1}$) is the large deviation rate function for the chordal SLE$_\kappa$ curve in $\uhp$ from $x_1$ to $x_2$ (resp.\ $x_2$ to $x_1$), when we take the $\kappa \to 0$ limit for this coupling conditioned on the points $\underline z$, the SLE$_\kappa$ chord $\gamma_1$ (resp.\ $\gamma_2$) should converge to the energy minimizing chord $\gamma_1^{\mathrm{min}}$ among those going from $x_1$ to $x_2$ (resp.\ $\gamma_2^{\mathrm{min}}$ from $x_2$ to $x_1$) in $\uhp$ while passing through $\underline z$. Given the uniqueness of $\gamma_1^{\mathrm{min}}$ and $\gamma_2^{\mathrm{min}}$ (which we do not claim but expect to be true for a general sequence $\underline z$), we know from their piecewise geodesic property (Proposition~\ref{prop:minimizer}) that these are orientation reversals of each other. That is, $\gamma_1$ and $\gamma_2$ are random perturbations of the same piecewise geodesic chord passing through $\underline z$.
The proof of Loewner energy reversibility in \cite{W1} essentially proceeds by taking $\underline z$ to become dense on $\gamma_1$ and $\gamma_2$, in which case the limits of $\gamma_1^{\mathrm{min}}$ and $\gamma_2^{\mathrm{min}}$ have the same trace due to the Jordan curve theorem.
\end{remark}

Here is an overview of the rest of this article. In Section~\ref{sec:preliminaries}, we define and discuss the properties of the partial Loewner energy of a simple curve, which can be considered as that of the chord obtained by completing it with a hyperbolic geodesic. We also collect some results about topologies on the space of chords. In Section~\ref{sec:proof}, we first establish an inequality version of the commutation relation for Loewner energy and then use it to prove Theorem~\ref{thm:main}. We prove Theorem~\ref{prop:isotopy} in Section~\ref{sec:isotopy}.

\vspace{8pt}

\textbf{Acknowledgements.} I wish to thank Yilin Wang for introducing this problem during her lectures at the Cornell Probability Summer School, and extend my gratitude to the organizers of the program. I thank Osama Abuzaid, Phil\'emon Bordereau, Luis Brummet, Liam Hughes, and Eveliina Peltola for helpful discussions, and Yilin Wang and Catherine Wolfram for their comments on the earlier version this article. I am grateful to Chlo\'e Postel-Vinay for raising the connection with braid groups. This work was completed in part during the Thematic Program on Randomness and Geometry at the Fields Institute, whose hospitality is acknowledged. The author is partially supported by a fellowship from Kwanjeong Educational Foundation.

\section{Preliminaries}\label{sec:preliminaries}

Given a simply connected domain $D\subsetneq \Chat = \C \cup \{\infty\}$ with distinct prime ends $a$ and $b$, we denote by $\mc X(D;a,b)$ the space of chords in $D$ (i.e., Jordan arcs $\gamma$ in $D$ with $\boundary \gamma \subset \partial D$) from $a$ to $b$. 
We consider these chords modulo orientation-preserving reparameterizations. In particular, $\mc X(D;a,b)$ and $\mc X(D;b,a)$ differ in the orientations of chords. 
Since our notation for Loewner energy already contains information about the orientation of the chord, we denote orientation reversals implicitly in the rest of this paper. That is, if $\gamma \in \mc X(D;a,b)$, then $I_{D;b,a}(\gamma)$ should be understood as the energy of the reversed chord $\gamma^\rev \in \mc X(D;b,a)$.

If $D$ is a bounded Jordan domain, we endow $\mc X(D;a,b)$ with the topology induced by the Hausdorff distance on the (closure of) chords. For a general simply connected domain $D$ conformally equivalent to the unit disk $\mathbb D:= \{|z| \leq 1\}$, we choose any conformal map $\varphi$ taking $D$ onto $\ud$ and endow $\mc X(D;a,b)$ with the pullback of the Hausdorff topology on $\mc X(\ud;\varphi(a),\varphi(b))$. We will make heavy use of the following fact that Loewner energy is lower semicontinuous with respect to the Hausdorff topology on chords.

\begin{lem}[\!\!{\cite[Lem.\ 2.7]{peltola_wang}}]\label{lem:lsc}
    The map $I_{D;a,b} : \mc X(D;a,b) \to [0,\infty]$ is lower semicontinuous, and its sublevel set $\{\gamma \in \mc X(D;a,b): I_{D;a,b}(\gamma) \leq c\}$ is compact for any $c \in [0,\infty)$. 
    
    In fact, if $(\gamma^n)_{n\in \mb N}$ is a sequence of chords in $\mc X(\ud;-1,1)$ with $\sup_{n\in \mb N} I_{\ud;-1,1}(\gamma^n) \leq c < \infty$, then there exists a corresponding sequence of homeomorphisms $\varphi_n:\overline{\ud} \to \overline{\ud}$, with $\varphi_n|_{\partial \ud}$ equal to the identity and $\gamma^n = \varphi_n([-1,1])$ for each $n\in \mb N$, such that $\varphi_n$ converges uniformly along a subsequence to some self-homeomorphism $\varphi$ of $\overline \ud$. Moreover, $\gamma:= \varphi([-1,1]) \in \mc X(\ud;-1,1)$ satisfies $I_{\ud;-1,1}(\gamma) \leq c$.
\end{lem}

The convergence in the latter half of the above lemma was established within its proof in \cite{peltola_wang}.

\subsection{Energy of a simple curve and hyperbolic geodesics}

In these subsection, we expand the definition of Loewner energy to curves that end in the interior of the domain and obtain its addtivity property \eqref{eq:additivity}. We then recall basic properties of hyperbolic geodesics that we use in our proof. The facts surveyed here can be found in \cite{W1}.

Let $\gamma: [0,T] \to \uhp \cup \{0\}$ be a Jordan arc (henceforth called a \textit{simple curve}) with $\gamma(0) = 0$ and $T \in (0,\infty)$. For each $t\in [0,T]$, we consider the unique conformal map $g_t: \uhp \setminus \gamma([0,t]) \to \uhp$ with \textit{hydrodynamic normalization}: i.e., satisfying 
\begin{equation} g_t(z) - z \to 0 \quad\text{as}\quad z\to \infty .\end{equation}
The \textit{half-plane capacity} of the curve $\gamma([0,t])$, which we denote $\hcap(\gamma([0,t]))$, is defined via the expansion 
\begin{equation}
    g_t(z) = z + \frac{\hcap(\gamma([0,t]))}{z} + O\bigg(\frac{1}{z^2}\bigg) \quad \text{as} \quad z\to\infty.
\end{equation}
We can reparametrize $\gamma$ on the interval $[0,\frac{1}{2}\hcap(\gamma)]$ so that $\hcap(\gamma([0,t])) = 2t$ for every $t$, which we call the \textit{(half-plane) capacity parameterization} of $\gamma$. For the rest of this subsection, we assume that $\gamma$ is in capacity parameterization with $T = \frac{1}{2}\hcap(\gamma)$.

We call the family of maps $(g_t)_{t\in [0,T]}$ the \textit{Loewner chain} corresponding to the curve $\gamma$, and define the \textit{driving function} of the curve $\gamma$ as the continuous function $t\mapsto \lambda_t:= g_t(\gamma(t))$. For every point $z\in \overline \uhp \setminus \{0\}$, the map $t\mapsto g_t(z)$ is continuously differentiable and satisfies the \textit{Loewner equation}
\begin{equation}\label{eq:loewner-eq}
    \partial_t g_t(z) = \frac{2}{g_t(z) - \lambda_t}
\end{equation}
on the entire interval $[0,T]$ if $z\notin \gamma$, or on the interval $[0,T_z)$ if $z = \gamma(T_z)$ for some $T_z \in (0,T]$.

\begin{df} \label{def:partial-uhp}
Given a simple curve $\gamma$ in $\uhp$ starting at 0, we define its \textit{Loewner energy targeted at $\infty$} as
\begin{equation}\label{eq:def-partial-uhp}
    I_{\uhp;0,\infty}(\gamma):= \frac{1}{2} \int_0^T \left| \frac{d\lambda_t}{dt}\right|^2 dt
\end{equation}
where $\lambda$ is the driving function of the curve $\gamma$ (in capacity parametrization). The integral above is set to be infinite if $\lambda_t$ is not absolutely continuous in $t$.
\end{df}

This extends the definition \eqref{eq:uhp-def} for the Loewner energy of a chord in $\mc X(\uhp;0,\infty)$, which amounts to the case $T=\infty$. For a simple curve $\gamma$ in an arbitrary simply connected domain $D$ starting from a prime end $a$, we define its Loewner energy targeted at a prime end $b$ so that it is conformally invariant: 
\begin{equation}\label{eq:partial-energy-conformal}
    I_{D;a,b}(\gamma) = I_{\uhp;0,\infty}(\varphi(\gamma))
\end{equation}
for any conformal map $\varphi: D\to \uhp$ with $\varphi(a) =0$ and $\varphi(b) = \infty$, analogously to \eqref{eq:energy-conformal-invariant}.

If $\gamma:[0,T]\to \uhp\cup \{0\}$ is a simple curve with $\gamma(0)=0$ in half-plane capacity parameterization, then, for each $s\in (0,T)$, the image $\gamma^s:= g_s(\gamma([s,T])) - \lambda_s$ is also a simple curve in $\uhp$ starting at 0. Its half-plane capacity parameterization is given by $\gamma^s(t) = g_s(\gamma(s+t))-\lambda_s$ on the interval $[0,T-s]$, and the corresponding driving function is $t\mapsto \lambda_{t+s} - \lambda_s$. We thus have
\begin{equation}\label{eq:additivity}
\begin{split}
    I_{\uhp;0,\infty}(\gamma) = \int_0^s \big| \dot \lambda_t \big|^2\,dt + \int_s^T \big| \dot \lambda_t \big|^2\,dt 
    &= I_{\uhp;0,\infty}(\gamma([0,s])) + I_{\uhp;0,\infty}(\gamma^s)\\ &= I_{\uhp;0,\infty}(\gamma([0,s])) + I_{\uhp \setminus \gamma([0,s]); \gamma(s),\infty}(\gamma([s,T]))
    \end{split}
\end{equation}
for each $s\in (0,T)$, as was observed in \cite{W1}. 
This additivity property \eqref{eq:additivity} can be stated for more general domains via conformal invariance \eqref{eq:partial-energy-conformal}, which naturally leads to the idea that the partial Loewner energy $I_{D;a,b}(\gamma_1)$ of a simple curve $\gamma_1$ in a simply connected domain $D$ starting at a prime end $a$ and ending at $\zeta \in D$ represents the Loewner energy of the chord $\gamma_1 \cup \gamma_2 \in \mc X(D;a,b)$ obtained by completing $\gamma_1$ all the way to the prime end $b$ by a chord $\gamma_2 \in \mc X(D\setminus \gamma_1; \zeta,b)$ with zero energy. Such zero energy chords are hyperbolic geodesics and they are central to the theory of Loewner energy.

\begin{df}
    The \textit{hyperbolic geodesic} in a simply connected domain $D\subsetneq \Chat$ from a prime end $a$ to another prime end $b$ is the image $\varphi(i\mathbb R_+)$ of the imaginary axis under a conformal map $\varphi:\mathbb H \to D$ with $\varphi(0)=a$ and $\varphi(\infty) = b$. Note that this specifies a unique chord in $\mc X(D;a,b)$.
\end{df}

For instance, the hyperbolic geodesic in $\uhp$ between $x,y\in \R$ is the semicircle with the interval between $x$ and $y$ as its diameter. The hyperbolic geodesic $\gamma$ in $D$ from $a$ to $b$ has the following useful properties that follow immediately from its definition.

\begin{itemize}
    \item The hyperbolic geodesic $\gamma$ is the unique chord in $\mc X(D;a,b)$ with $I_{D;a,b}(\gamma) = 0$. Any other chord in $\mc X(D;a,b)$ has positive (or infinite) Loewner energy $I_{D;a,b}$.

    \item The reversed chord $\gamma^\rev$ is the hyperbolic geodesic in $D$ from $b$ to $a$. Hence, it makes sense to omit the orientation and call $\gamma$ the hyperbolic geodesic in $D$ between $a$ and $b$.

    \item Given any parameterization $\gamma:(0,T)\to D$ with $\gamma(0) = a$ and $\gamma(T) = b$, for each $t\in (0,T)$, the curve $\gamma(t,T)$ is the hyperbolic geodesic in $D\setminus \gamma([0,t])$ between $\gamma(t)$ and $b$.
\end{itemize}

Let us say that a chord $\gamma \in \mc X(D;a,b)$ has a \textit{geodesic tip} if we can decompose it into simple curves $\gamma_1$ and $\gamma_2$ joined at $\zeta \in D$ such that $\gamma_2$ is the hyperbolic geodesic in $D\setminus \gamma_1$ between $\zeta$ and $b$. (Note that by the third bulleted observation above, a chord $\gamma \in \mc X(D;a,b)$ has a geodesic tip if and only if, given any parameterization $\gamma:(0,T)\to D$ with $\gamma(0)=a$ and $\gamma(T)=b$, there exists a time $s \in (0,T)$ such that the following is true: for every $t \in [s,T)$, the curve $\gamma([t,T])$ is the hyperbolic geodesic in $D\setminus \gamma([0,t])$ between $\gamma(t)$ and $b$.) Then,
\begin{equation}
    I_{D;a,b}(\gamma) = I_{D;a,b}(\gamma_1) + I_{D\setminus \gamma;\zeta,b}(\gamma_2) = I_{D;a,b}(\gamma_1) = I_{D;a,b}(\gamma_1) + I_{D\setminus \gamma_1; b,\zeta}(\gamma_2)
\end{equation}
by \eqref{eq:additivity} and the first two bulleted observations above. Here, we already see that when computing the Loewner energy of a chord with a geodesic tip, we can reverse the orientation of this tip. The the key idea in our proof is to generalize this to the case that the hyperbolic geodesic is in a middle part of the chord (see Lemma~\ref{lem:comm-ineq}).

\subsection{Identifying chords via hulls}

In our proof of the ``commutation inequality" (Lemma~\ref{lem:comm-ineq}), we will need to identify the Hausdorff limit of chords using the convergence of the hulls that they generate. 
For concreteness, let us consider chords in the upper half-plane $\uhp$ between two points on $\R$. Given $\gamma \in \mc X(\uhp;x,y)$, by the Jordan curve theorem, $\uhp \setminus \gamma$ consists of exactly two connected components: a bounded component $D_\gamma$ and an unbounded component $U_\gamma$, both of which are simply connected domains. Moreover, $\partial D_\gamma = \gamma \cup [x,y]$ where $[x,y]\subset \R$ denotes the closed interval between $x$ and $y$ even when $x>y$. We define the \textit{hull} generated by the chord $\gamma$ as the closure $K_\gamma:= \overline{D_\gamma}$ in $\overline \uhp = \uhp \cup \R$. Since $\partial K_\gamma \cap \uhp = \partial D_\gamma \cap \uhp = \gamma$, a bounded chord in $\uhp$ can be uniquely identified by the hull it generates.

\begin{lem}\label{lem:hull-compactness}
    Suppose $(\gamma^n)_{n\in \mb N}$ is a sequence of chords in $\uhp$ between two fixed points $x,y\in \R$ such that $\sup_{n\in \mb N} I_{\uhp;x,y}(\gamma^n) < \infty$. Then, there exists a subsequence $\gamma^{n_k}$ which converges to some chord $\gamma \in \mc X(\uhp;x,y)$ in the Hausdorff distance and, furthermore, the corresponding hulls $K_{\gamma^{n_k}}$ converge to $K_\gamma$ in the Hausdorff distance.
\end{lem}
\begin{proof}
    Without loss of generality, assume $x<y$.
    Fix a M\"obius map $f: \ud \to \mb H$ with $f(-1) = x$ and $f(1)=y$. By Lemma~\ref{lem:lsc}, for each $n$, there exists a homeomorphism $\varphi_n:\overline{\ud} \to \overline {\ud}$ such that $\varphi_n|_{\partial \ud}$ is the identity map and $\varphi_n([-1,1]) = f^{-1}(\gamma_n)$. Furthermore, we can find a subsequence $\varphi_{n_k}$ which converges uniformly to a homeomorphism $\varphi:\overline{\ud} \to \overline{\ud}$. Then, $\gamma_{n_k}$ converges to $\gamma:= (f\circ \varphi)([-1,1])\in \mc X(\uhp;x,y)$ in the Hausdorff distance.
    
    Denote by $D$ the lower half-disk $\{z\in \ud: \mathrm{Im}(z)<0\}$ and note that the bounded connected component of $\uhp \setminus \gamma_n$ is $(f\circ \varphi_n)(D)$. Since the homeomorphisms $f\circ \varphi_{n_k}:\overline{\ud}\to \overline{\mb H}\cup \{\infty\}$ converge uniformly, the hulls $K_{\gamma_{n_k}} = \overline{(f\circ\varphi_{n_k})({D_\eta})} = (f\circ \varphi_{n_k})(\overline{D})$ also converge in the Hausdorff distance to $K_\gamma = \varphi(\overline{D})$.
\end{proof}

Our definition of the hull $K_\gamma$ generated by a bounded chord $\gamma$ in $\uhp$ is in agreement with the following more general definition of a hull in $\uhp$.
Let $\mc C$ be the space of non-empty closed subsets of $\overline \uhp$. The set of \textit{compact $\uhp$-hulls} is defined as
\begin{equation*}
    \mc K := \{K \in \mc C: \text{$K$ is bounded, $\uhp\setminus K$ is simply connected, and $\overline{K\cap \uhp} = K$}\}.
\end{equation*}
Indeed, since $D_\gamma$ is bounded, $K_\gamma$ is bounded. Moreover, $K_\gamma \cap \uhp = D_\gamma \cup (\partial D_\gamma \cap \uhp) = D_\gamma \cup \gamma$, so $\uhp \setminus K_\gamma = (\uhp \setminus \gamma)\setminus D_\gamma = U_\gamma$, which is simply connected. Also, $\overline{K_\gamma \cap \uhp} = \overline{D_\gamma} \cup \overline{\gamma} = K_\gamma \cup (\gamma \cup \{x,y\}) = K_\gamma$. Thus, the hull $K_\gamma$ defined as the closure of the bounded component of $\uhp \setminus \gamma$ in $\overline \uhp$ is an element of $\mc K$.

For each hull $K \in \mc K$, there exists a unique conformal transformation $g_K: \uhp \setminus K \to \uhp$ such that $g_K(z) -z \to 0$ as $z\to \infty$, which we refer to as the \textit{mapping-out function} of $K$. This induces the following alternative topology on $\mc K$ which is natural from the perspective of Loewner chains.

\begin{df}\label{def:cara}
Let $(K^n)_{n\in\mb N}$ be a sequence compact $\uhp$-hulls. 
We say that $K^n$ converges to $K\in \mc K$ in the \textit{Carath\'eodory topology} if $g_{K^n}^{-1}$ converges to $g_K^{-1}$ uniformly away from $\R$. This means that for every $\ep>0$, $g_{K^n}^{-1}$ converges uniformly to $g_K^{-1}$ on $\{z\in \C: \Im z > \ep\}$. 
\end{df}

An equivalent geometric definition is available via Carath\'eodory kernel theorem (see, e,g., \cite[Sec.\ 3.3]{Beliaevs_book}), from which the following topological lemma follows. This allows us to identify the limiting chord in Lemma~\ref{lem:lsc} through the Carath\'eodory limit of the hulls.

\begin{lem}[\!{\cite[Lem.\ 2.3]{peltola_wang}}]
    \label{lem:same-limit}
    Let $(K^n)_{n \in \mb N}$ be a sequence in $\mc K$. If $K^n$ converges to $K \in \mc K$ with respect to the Carath\'eodory topology and to a compact set $K^* \subset \overline\uhp$ in the Hausdorff distance, then $\uhp \setminus K$ coincides with the unbounded connected component of $\uhp \setminus K^*$. In particular, if $K^* \in \mc K$, then $K = K^*$.
\end{lem}

\section{Proof of reversibility of chordal Loewner energy}\label{sec:proof}

In this section, we first establish the inequality version of the commutation relation for Loewner energy when two curves are joined at the same endpoint (Lemma~\ref{lem:comm-ineq}). Based on this result, we show that a minimal energy chord passing through a fixed finite set of points is piecewise geodesic, and the Loewner energy of any piecewise geodesic chord is invariant under reversing its orientation (Proposition~\ref{prop:minimizer}). We use this to give a proof of Theorem~\ref{thm:main} at the end of this section.

\subsection{A commutation inequality}

The starting point of our proof of Theorem~\ref{thm:main} is the following commutation relation for Loewner energy given in \cite[Cor.\ 4.4]{W1}.

\begin{thmx}\label{thm:commutation}
    Let $D$ be a simply connected domain and $a,b$ be distinct prime ends. Suppose $\gamma$ is a simple curve $D$ from $a$ to $\gamma(S) \in D$, and $\eta$ is a simple curve in $D$ from $b$ to $\eta(T) \in D$ with $\gamma\cap \eta = \varnothing$. Then,
    \begin{equation}\label{eq:communtation}
        I_{D;a,b}(\gamma) + I_{D\setminus\gamma ;b,\gamma(S)}(\eta) = I_{D;b,a}(\eta) + I_{D \setminus \eta;a,\eta(T)}(\gamma).
    \end{equation}
\end{thmx}

In other words, we can compute the joint Loewner energy of $\gamma$ and $\eta$ by computing the energy of either $\gamma$ or $\eta$ first, then adding the energy of the other curve in the complement of the first.
The main restriction in Theorem~\ref{thm:commutation} is that the curves $\gamma$ and $\eta$ cannot intersect, even at their endpoints. 
Nevertheless, we can deduce from it a ``commutation inequality" satisfied by two curves terminating at the same interior point.

\begin{lem}\label{lem:comm-ineq}
    Let $D$ be a simply connected domain with distinct prime ends $a$ and $b$. Suppose $\gamma$ and $\eta$ are disjoint simple curves in $D$ starting from $a$ and $b$, respectively, except that they both terminate at $\zeta \in D$. Then,
    \begin{equation}\label{eq:comm-ineq-1}
        I_{D;b,a}(\eta) \leq I_{D;a,b}(\gamma) + I_{D\setminus \gamma;b,\zeta}(\eta) .
    \end{equation}
    Furthermore, if $\eta$ has a geodesic tip when considered as a chord in $D\setminus\gamma$ from $b$ to $\zeta$,\footnote{That is, $\eta$ can be partitioned into simple curves $\eta_1$, from $b$ to an intermediate point $\xi \in \eta$, and $\eta_2$, from $\xi$ to $\zeta$, such that $\eta_2$ is the hyperbolic geodesic in $D\setminus (\gamma \cup \eta_1)$ between $\xi$ and $\zeta$.} then
    \begin{equation}\label{eq:comm-ineq}
        I_{D;b,a}(\eta) + I_{D\setminus \eta;a,\zeta}(\gamma) \leq I_{D;a,b}(\gamma) + I_{D\setminus \gamma;b,\zeta}(\eta) .
    \end{equation}
\end{lem}

\begin{proof}
    Take any continuous parameterization $\eta:(0,T] \to D$ with $\eta(0)=b$ and $\eta(T) = \zeta$. For $t\in (0,T)$, since $\gamma$ and $\eta((0,t])$ do not intersect, the commutation relation (Theorem~\ref{thm:commutation}) states
    \begin{equation}\label{eq:comm-ineq-2}
        I_{D;b,a}(\eta([0,t])) + I_{D\setminus \eta([0,t]);a,\eta(t)}(\gamma) = I_{D;a,b}(\gamma) + I_{D\setminus \gamma;b,\zeta}(\eta([0,t])).
    \end{equation}
    Note from the definition of the Loewner energy of a curve \eqref{eq:def-partial-uhp}--\eqref{eq:partial-energy-conformal} that
    \begin{equation}
        \lim_{t \to T^-} I_{D;b,a}(\eta([0,t])) = I_{D; b,a}(\eta) \quad \text{and} \quad \lim_{t\to T^-} I_{D\setminus \gamma;b,\zeta}(\eta([0,t])) = I_{D\setminus \gamma; b,\zeta}(\eta).
    \end{equation}
    We thus obtain
    \begin{equation}\label{eq:comm-ineq-3}
        \lim_{t\to T^-} I_{D\setminus \eta([0,t]);a,\eta(t)}(\gamma) = I_{D;a,b}(\gamma) + I_{D\setminus \gamma;b,\zeta}(\eta) - I_{D;b,a}(\eta).
    \end{equation}
    The first inequality \eqref{eq:comm-ineq-1} follows from the simple observation that the left-hand side of \eqref{eq:comm-ineq-3} must be nonnegative by the definition of Loewner energy.
    
    If we furthermore had  
    \begin{equation}\label{eq:comm-ineq-4}
        I_{D\setminus\eta;a,\zeta}(\gamma) \leq \lim_{t\to T^-} I_{D\setminus \eta([0,t]);a,\eta(t)}(\gamma),
    \end{equation}
    then we would obtain the second inequality \eqref{eq:comm-ineq} from \eqref{eq:comm-ineq-3}. 
    For the rest of this proof, we prove the inequality \eqref{eq:comm-ineq-4} in the case that $\eta \in \mc X(D\setminus \gamma;b,\zeta)$ has a geodesic tip. That is, suppose there exists $T_0 \in [0,T)$ such that $\eta([T_0,T])$ is the hyperbolic geodesic in $D \setminus (\gamma \cup \eta([0,T_0]))$ between $\eta(T_0)$ and $\zeta = \eta(T)$. By the conformal invariance \eqref{eq:partial-energy-conformal} of Loewner energy, we may assume that $D = \uhp$, $a=1$, and $b=0$. Let us further assume that $\eta$ is in half-plane capacity parameterization and let $(g_t)_{t\in [0,T]}$ be the Loewner chain corresponding to the curve $\eta$. 
    \begin{figure}
        \centering
        \includegraphics[width=.8\linewidth]{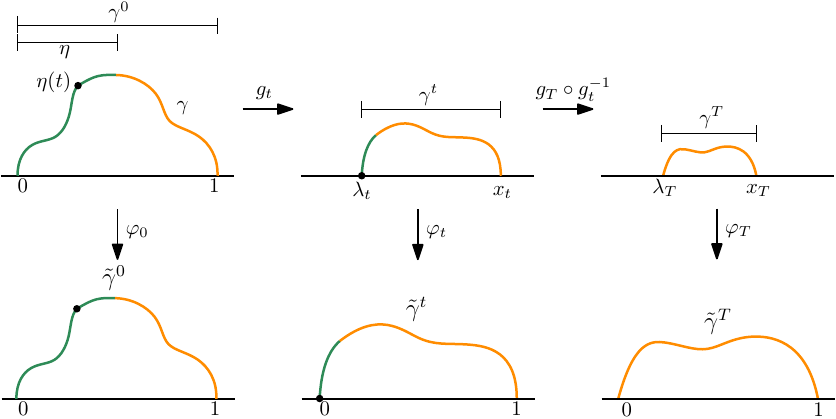}
        \caption{Proof of Lemma~\ref{lem:comm-ineq}. Let $(g_t)_{t\in [0,T]}$ be the Loewner chain mapping out $\eta$, and denote by $\gamma^t$ the chord remaining after mapping out $\gamma \cup \eta$ by $g_t$. Let $\tilde \gamma^t$ be the image of $\gamma^t$ under an affine transformation that maps the endpoints of the chord to 0 and 1. The hulls enclosed by $\gamma^t$ converges in the Carath\'eodory topology to that enclosed by $\gamma^T$, and the chords $\tilde \gamma^t$ converge along a subsequence to some $\tilde \gamma^*\in \mc X(\uhp;1,0)$ in Hausdorff distance. We have $\tilde \gamma^* = \varphi_T(\gamma^T)$ from Lemma~\ref{lem:same-limit}. }
        \label{fig:comm-ineq}
    \end{figure}
    As illustrated in Figure~\ref{fig:comm-ineq}, let us denote
    \begin{equation*}
        \gamma^t:= g_t(\gamma \cup \eta([T,t])) \in \mc X(\uhp;x_t,\lambda_t)
    \end{equation*}
    where $\lambda_t:= g_t(\eta(t))$ is the driving function of the Loewner chain and $x_t:= g_t(1)$. Note that $x_t$ and $\lambda_t$ are continuous real-valued functions satisfying $\lambda_t \neq x_t$ for $t\in [0,T]$. For each $t$, let $\varphi_t(z) = (z-\lambda_t)/(x_t-\lambda_t)$ be the conformal automorphism of $\uhp$ sending the triple $(\lambda_t,x_t,\infty)$ to $(0,1,\infty)$ and denote
    \begin{equation*}\tilde \gamma^t:= \varphi_t(\gamma^t) \in \mc X(\uhp;0,1) .\end{equation*} 
    For $t\in [T_0,T]$, since $\eta([T,t])$ is the hyperbolic geodesic in $D\setminus (\gamma \cup \eta([0,t]))$ between $\zeta = \eta(T)$ and $\eta(t)$, we have  
    \begin{equation}\begin{split} \label{eq:comm-ineq-5}
        I_{\uhp;1,0}(\tilde \gamma^t) = I_{\uhp;x_t,\lambda_t}(\gamma^t) 
        &= I_{\uhp \setminus \eta([0,t]);1,\eta(t)}(\gamma \cup \eta([T,t])) \\& = I_{\uhp \setminus \eta([0,t]);1,\eta(t)}(\gamma) + I_{\uhp\setminus (\gamma \cup \eta([0,t]));\zeta,\eta(t)}(\eta([T,t]))
        \\& = I_{\uhp \setminus \eta([0,t]);1,\eta(t)}(\gamma).\end{split}
    \end{equation}
    From \eqref{eq:comm-ineq-2}, we have
    \begin{equation} I_{\uhp;0,1}(\tilde \gamma^t)= I_{\uhp \setminus \eta([0,t]);1,\eta(t)}(\gamma) \leq I_{\uhp;1,0}(\gamma) + I_{\uhp \setminus \gamma;0,\zeta}(\eta) =: c\end{equation}
    for all $t \in [T_0,T]$. The lemma holds trivially if $c=\infty$, so let us assume $c<\infty$. By Lemma~\ref{lem:lsc}, there exists a sequence $(t_k)_{k\in \mathbb N}$ increasing to $T$ and a chord $\tilde \gamma^* \in \mc X(\uhp;1,0)$ such that $\tilde \gamma^{t_k} \to \tilde \gamma^*$ and $K_{\tilde \gamma^{t_k}} \to K_{\tilde \gamma^*}$ in the Hausdorff distance. 

    We claim that $\tilde \gamma^* = \tilde \gamma^T$, for which it suffices to show $\gamma^*:= \varphi_T^{-1}(\tilde \gamma^*) = \gamma^T$. Recalling that $x_{t_k} \to x_T$ and $\lambda_{t_k} \to \lambda_T$, the affine maps $\varphi_{t_k}$ converge uniformly on bounded subsets of $\overline \uhp$ to $\varphi_T$. Thus, the rescaled hulls $K_{\gamma^{t_k}} = \varphi_{t_k}^{-1} (K_{\tilde \gamma^{t_k}})$ converge to $K_{\gamma^*} = \varphi_T^{-1}(K_{\tilde \gamma^*})$ in the Hausdorff distance. On the other hand, observe that if $f_T: \uhp \setminus K_{\gamma^T} \to \uhp$ is the mapping-out function of the hull $K_{\gamma^T}$ (i.e., the unique hydrodynamically normalized conformal map), then $f_t:= f_T \circ g_T \circ g_t^{-1}$ is the mapping-out function of the hull $K_{\gamma^t}$. We see from the Loewner equation \eqref{eq:loewner-eq} that $g_t \to g_T$ uniformly away from $\R$ as $t\to T^-$. Since $f_T \circ g_T$ is hydrodynamically normalized, we conclude that $f_t^{-1} \to f_T^{-1}$ uniformly away from $\R$ and thus $K_{\gamma^t} \to K_{\gamma^T}$ in the Carath\'eodory topology as $t\to T^-$. By Lemma~\ref{lem:same-limit}, $K_{\gamma^*} = K_{\gamma^T}$ and hence $\gamma^* = \partial K_{\gamma^*} \cap \uhp  = \partial K_{\gamma^T} \cap \uhp = \gamma^T$ as claimed. In particular,
    \begin{equation*}
        I_{\uhp;1.0}(\tilde \gamma^*) = I_{\uhp;1,0}(\gamma^T) = I_{\uhp\setminus \eta;1,\zeta}(\gamma).
    \end{equation*}
    Finally, by Lemma~\ref{lem:lsc},
    \begin{equation}
         I_{\uhp\setminus \eta;1,\zeta}(\gamma) = I_{\uhp;1,0}(\tilde \gamma^*) \leq \liminf_{k\to\infty} I_{\uhp;1,0}(\tilde \gamma^{t_k}) = \lim_{t\to T^-} I_{\uhp;1,0}(\tilde \gamma^t) = \lim_{t\to T^-} I_{\uhp \setminus \eta([0,t]);1,\eta(t)}(\gamma).
    \end{equation}
    This conclusion is equivalent to \eqref{eq:comm-ineq-4} by the conformal invariance of Loewner energy.
    \end{proof}

\subsection{Piecewise geodesic energy minimizers}

Let $D$ be a simply connected domain with two distinct prime ends $a$ and $b$. Recall that $\mc X(D;a,b)$ is the set of all chords in $D$ from $a$ to $b$. Given a finite set $\underline z \subset D$, let 
\begin{equation}
    \mc X(D;a,b;\underline z) = \{\gamma \in \mc X(D;a,b): \underline z \subset \gamma\}
\end{equation}
be the set of chords in $D$ from $a$ to $b$ passing through all points in $\underline z$. We first observe that there exists a Loewner energy minimizer on this set.

\begin{lem}\label{lem:existence}
    Given any finite $\underline z \subset D$, there exists a chord $\gamma^{\mathrm{min}} \in \mc X(D;a,b;\underline z)$ such that 
    \begin{equation}
        I_{D;a,b}(\gamma^{\mathrm{min}}) := \inf_{\gamma \in \mc X(D;a,b;\underline z)}I_{D;a,b}(\gamma) < \infty.
    \end{equation}
\end{lem}
\begin{proof}
    Without loss of generality, let $D$ be a bounded Jordan domain.
    The existence of a finite-energy chord in $\mc X(D;a,b;\underline z)$ is given in \cite[Lem.\ 3.3]{W1}. Let us thus take a sequence of finite-energy chords $(\gamma^n)_{n\in \mb N}$ in $\mc X(D;a,b;\underline z)$ whose energies decrease to $\inf_{\gamma \in \mc X(D;a,b;\underline z)}I_{D;a,b}(\gamma)<\infty$. By Lemma~\ref{lem:lsc}, there exists a subsequence $\gamma^{n_k}$ which converges in the Hausdorff distance to some chord $\gamma^{\mathrm{min}} \in \mc X(D;a,b)$. We furthermore have $I_{D;a,b}(\gamma^{\mathrm{min}}) \leq \liminf_{k\to\infty} I_{D;a,b}(\gamma^{n_k})$. Observe that $\max_{z \in \underline z} \min_{w \in \gamma^{\mathrm{min}}} |z-w|$ is bounded above by the Hausdorff distance between $\gamma^{n_k}$ and $\gamma^{\mathrm{min}}$ for every $n_k$. Since the latter converges to 0 as $k\to\infty$, we have $\underline z \subset \gamma^{\mathrm{min}}$.
\end{proof}

The key observation in our proof of Theorem~\ref{thm:main} is that the energy minimizer in Lemma~\ref{lem:existence} has the following property, which was introduced in \cite{piecewise-geodesic-i} for the case of Jordan curves.

\begin{df}
Let $\underline z \subset D$ be a finite set. We say that a chord $\gamma \in \mc X(D;a,b;\underline z)$ is \textit{piecewise geodesic relative to $\underline z$} if $\underline z$ partitions $\gamma$ into simple curves $\gamma_1,\gamma_2,\dots,\gamma_n$ such that each $\gamma_k$ is a hyperbolic geodesic of the domain $D \setminus (\gamma\setminus \gamma_k)$.
\end{df}

In other words, a chord $\gamma:[0,1]\to D\cup\{a,b\}$ visiting $\underline z \cup \{a,b\}$ at times $0 = t_0 < t_1 <\dots < t_{n-1} < t_n = 1$ is piecewise geodesic relative to $\underline z$ if, for every $k=1,\dots,n$, the chord $\gamma([t_{k-1},t_k])$ is the hyperbolic geodesic in $D\setminus \gamma([0,t_{k-1}]\cup [t_k,1])$ between $\gamma(t_{k-1})$ and $\gamma(t_k)$.

\begin{prop} \label{prop:minimizer}
    Let $D$ be a simply connected domain with distinct prime ends $a,b$ and let $\underline z$ be a finite subset of $D$.
    \begin{enumerate}
        \item If $\gamma$ is a chord in $\mc X(D;a,b;\underline z)$ satisfying $I_{D;a,b}(\gamma) = \inf_{\eta \in \mc X(D;a,b,\underline z)} I_{D;a,b}(\eta)$, then $\gamma$ is piecewise geodesic relative to $\underline z$. 
        
        \item If $\gamma \in \mc X(D;a,b;\underline z)$ is piecewise geodesic relative to $\underline z$, then its energy is invariant under the reversal of orientation. That is, $I_{D;a,b}(\gamma) = I_{D;b,a}(\gamma)$.
    \end{enumerate}
    
\end{prop}

\begin{proof}
    Take an arbitrary chord $\gamma \in \mc X(D;a,b;\underline z)$ and let $z_1,z_2,\dots,z_n$ be the enumeration of $\underline z$ in the order visited by $\gamma$. For each $k=1,2,\dots,n+1$, let us denote the part of the chord $\gamma$ between $z_{k-1}$ and $z_k$ as $\gamma_k$ (where $z_0=a$ and $z_{n+1}=b$). We shall reverse the orientation of $\gamma$ inductively, finishing with a chord $\eta \in \mc X(D;b,a;\underline z)$ with $I_{D;b,a}(\eta) \leq I_{D;a,b}(\gamma)$. Furthermore, we will show that the inequality is strict if $\gamma$ is not piecewise geodesic relative to $\underline z$. Let us denote $D_k:= D\setminus(\bigcup_{i \leq k} \gamma_i)$ for the rest of this proof.

    \begin{itemize}
        \item In the first step, replace $\gamma_{n+1}$ with the hyperbolic geodesic in $D_n$ from $b$ to $z_n$, which we denote $\eta_{n+1}$. Since $I_{D_n;b,z_n}(\eta_{n+1})=0$, we have
        \begin{equation}\begin{split}\label{eq:rev-1}
            I_{D;a,b}(\gamma) &= I_{D;a,z_n}\Bigg( \bigcup_{i\leq n} \gamma_i \Bigg) + I_{D_n;z_n,b}\big(\gamma_{n+1}\big)\\
            &\geq I_{D;a,z_n}\Bigg( \bigcup_{i\leq n} \gamma_i \Bigg) + I_{D_n;b,z_n}\big(\eta_{n+1}\big).
            \end{split}
        \end{equation}
        Here, the inequality is strict if $\gamma_{n+1}$ is not a hyperbolic geodesic (i.e., $\gamma_{n+1}$ is not the orientation reversal of $\eta_{n+1}$). 
        Note that $(\bigcup_{i\leq n} \gamma_i) \cup \eta_{n+1} \in \mc X(D;a,b;\underline z)$.

        \item \begin{figure}
            \centering
            \includegraphics[width=\linewidth]{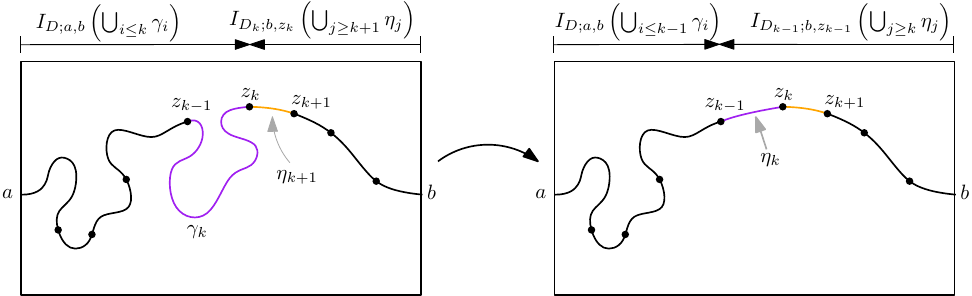}
            \caption{An orientation reversal step in the proof of Proposition~\ref{prop:minimizer}. Left: The curves $\gamma_{k+1},\dots,\gamma_{n+1}$ to the right of $z_k$ have been replaced with $\eta_{k+1},\dots,\eta_{n+1}$ in the previous steps. The orange curve $\eta_{k+1}$ is a hyperbolic geodesic in the complement of black and purple parts of the chord. Right: We replace the purple curve $\gamma_k$ with a hyperbolic geodesic $\eta_k$ in the complement of black and orange parts of the chord. The total energy of the curves in the right figure is less than or equal to that in the left figure. Moreover, if $\gamma_k \neq \eta_k$, then the inequality is strict. }
            \label{fig:geodesic-replacement}
        \end{figure}
        
        Let $k \in \{1,\dots,n\}$. Suppose we have replaced each $\gamma_j$ satisfying $j>k$ with a curve $\eta_j$ between $z_{j-1}$ and $z_j$ so that $(\bigcup_{i\leq k} \gamma_i) \cup (\bigcup_{j\geq k+1}\eta_j)$ is a chord in $\mc X(D;a,b;\underline z)$. Suppose, furthermore, that $\eta_{k+1}$ is the hyperbolic geodesic in $D_k \setminus (\bigcup_{j \geq k+2} \eta_j)$ between $z_{k+1}$ and $z_{k+2}$, and
        \begin{equation}\label{eq:rev-2}
            I_{D;a,b}(\gamma) \geq I_{D;a,z_k}\Bigg( \bigcup_{i\leq k} \gamma_i \Bigg) + I_{D_k;b,z_k}\Bigg(\bigcup_{j\geq k+1} \eta_j\Bigg).
        \end{equation}
        Now, replace $\gamma_k$ with the hyperbolic geodesic in $D_{k-1} \setminus (\bigcup_{j \geq k+1} \eta_j)$ from $z_k$ to $z_{k-1}$, which we denote $\eta_k$; see Figure~\ref{fig:geodesic-replacement} for an illustration.
        Since $\bigcup_{j\geq k+1}\eta_j$ has a geodesic tip as a chord in $D_k$, Lemma~\ref{lem:comm-ineq} implies
        \begin{equation}\begin{split}\label{eq:rev-3}
            &I_{D_{k-1};z_{k-1},z_k}\big(\gamma_k\big) + I_{D_{k-1}\setminus \gamma_k;b,z_k}\Bigg(\bigcup_{j\geq k+1}\eta_j\Bigg) \\ 
            &\geq I_{D_{k-1};b,z_{k-1}}\Bigg(\bigcup_{j\geq k+1}\eta_j\Bigg) + I_{D_{k-1} \setminus (\bigcup_{j\geq k+1}\eta_j);z_{k-1},z_k}\big(\gamma_k\big)\\
            &\geq I_{D_{k-1};b,z_{k-1}}\Bigg(\bigcup_{j\geq k+1}\eta_j\Bigg) + I_{D_{k-1}\setminus (\bigcup_{j\geq k+1}\eta_j);z_k,z_{k-1}}\big(\eta_k\big) 
            = I_{D_{k-1};b,z_{k-1}}\Bigg(\bigcup_{j\geq k}\eta_j\Bigg).
        \end{split}\end{equation}
        Note that the second inequality in \eqref{eq:rev-3} is strict if $\gamma_k$ is not a hyperbolic geodesic (i.e., $\gamma_k$ is not the orientation reversal of $\eta_k$).
        The new chord $(\bigcup_{i\leq k-1}\gamma_k) \cup (\bigcup_{j\geq k} \eta_j)$  satisfies the induction hypothesis since it passes through all points in $\underline z$ and, combining \eqref{eq:rev-2} and \eqref{eq:rev-3}, we have
        \begin{equation}
            \begin{split}\label{eq:rev-4}
                I_{D;a,b}\big(\gamma\big) &\geq I_{D;a,z_k}\Bigg( \bigcup_{i\leq k} \gamma_i \Bigg) + I_{D_k;b,z_k}\Bigg(\bigcup_{j\geq k+1} \eta_j\Bigg)\\
                &= I_{D;a,z_{k-1}}\Bigg( \bigcup_{i\leq k-1} \gamma_i \Bigg) + I_{D_{k-1};z_{k-1},z_k}\big(\gamma_k\big) + I_{D_k;b,z_k}\Bigg(\bigcup_{j\geq k+1}\eta_j\Bigg)\\
                &\geq I_{D;a,z_{k-1}}\Bigg( \bigcup_{i\leq k-1} \gamma_i \Bigg) + I_{D_{k-1};b,z_{k-1}}\Bigg(\bigcup_{j\geq k}\eta_j\Bigg).
            \end{split}
        \end{equation}
    \end{itemize}
    After completing this above reversal procedure all the way from $k = n+1$ to $k=1$, we are left with a chord $\eta:= \bigcup_{j=1}^{n+1}\eta_j \in \mc X(D;a,b;\underline z)$ satisfying $I_{D;b,a}(\eta) \leq I_{D;a,b}(\gamma)$. 
    By construction, if $\gamma$ is not piecewise geodesic relative to $\underline z$, then we cannot have $\eta_k = \gamma_k$ modulo orientation for all $k$. In this case, we have a strict inequality $I_{D;b,a}(\eta) < I_{D;a,b}(\gamma)$. 
    Applying the above inductive reversal procedure again to $\eta$, we obtain a chord $\tilde \gamma \in \mc X(D;a,b;\underline z)$ such that $I_{D;a,b}(\tilde \gamma) \leq I_{D;b,a}(\eta) < I_{D;a,b}(\gamma)$. Hence, if $\gamma$ minimizes the energy $I_{D;a,b}$ among chords passing through $\underline z$, then it must be piecewise geodesic relative to $\underline z$.

    If $\gamma$ is piecewise geodesic relative to $\underline z$, then $\eta_k = \gamma_k$ modulo orientation in each step of the above reversal procedure. Then, the chord $\eta$ that we obtain through the reversal procedure is simply $\gamma$ in reverse orientation and we have $I_{D;b,a}(\gamma) \leq I_{D;a,b}(\gamma)$. Note that the piecewise geodesic property of $\gamma$ is invariant under reversing its orientation, whence $I_{D;b,a}(\gamma) \geq I_{D;a,b}(\gamma)$ as well. We conclude that $I_{D;a,b}(\gamma) = I_{D;b,a}(\gamma)$ if $\gamma$ is piecewise geodesic.
\end{proof}

    We note that the loop version of Proposition~\ref{prop:minimizer} was given in \cite[Prop.\ 2.13]{RW} using the root-invariance of loop Loewner energy. The latter fact was proved in the same work using the the chordal reversibility as a preliminary.

The final input needed for our proof of Theorem~\ref{thm:main} is the following consequence of the Jordan curve theorem.

\begin{lem}\label{lem:chord-inclusion}
    Let $D \subsetneq \mb C$ be a bounded Jordan domain and suppose $\gamma$ and $\tilde \gamma$ are chords in $D$ between distinct boundary points $a$ and $b$. If $\gamma \subseteq \tilde \gamma$, then $\gamma = \tilde \gamma$.
\end{lem}
\begin{proof}
    Suppose, for contradiction, that there is a point $z \in \tilde \gamma \setminus \gamma$. It must be in $D$ since $\tilde \gamma \setminus D = \gamma \setminus D = \{a,b\}$. By the Jordan curve theorem, $D\setminus \gamma$ has two connected components, say $U$ and $V$, such that $D \cap \partial U = D \cap \partial V = \gamma \setminus \{a,b\}$ (see, e.g., \cite[Prop.\ 2.12]{Pommerenke_boundary} for a proof). Let $U$ be the component containing $z$. 
    
    Now, $D\setminus \tilde \gamma$ also has exactly two connected components, say $\widetilde U$ and $\widetilde V$. Since $D\setminus \tilde \gamma \subset D\setminus \gamma$, the two components $\widetilde U$ and $\widetilde V$ are each a subset of either $U$ or $V$ but not both. However, $z \in \tilde \gamma = D \cap \partial \widetilde U = D \cap \partial \widetilde V$, which means that both $\widetilde U$ and $\widetilde V$ are subsets of $U$. This implies $V \subset \tilde \gamma$ and hence $V \subset D \cap \partial \widetilde U$, which is a contradiction since $V \cap (D \cap \partial \widetilde U) \subset V \cap (D \cap \overline U) = V \cap (U \cup \gamma) = \varnothing$.
\end{proof}

We are now ready to prove the orientation reversibility of chordal Loewner energy.

\begin{proof}[Proof of Theorem~\ref{thm:main}]
    Without loss of generality, assume that $D$ is a bounded Jordan domain.
    Suppose $\gamma$ is a chord in $D$ from $a$ to $b$ with $I_{D;a,b}(\gamma) < \infty$. Choose an increasing sequence $\underline z^1 \subset \underline z^2 \subset \cdots $ of finite subsets of $\gamma \cap D$ such that $\bigcup_{n\geq 1} \underline z^n$ is dense in $\gamma$. For instance, we can take any continuous parameterization $\gamma:[0,1] \to D \cup \{a,b\}$ and set $\underline z^n = \{\gamma(j/2^n): j = 1,\dots,2^n-1\}$. By Lemma~\ref{lem:existence}, for each positive integer $n$, we can find chord $\gamma^n \in \mc X(D;a,b;\underline z^n)$ which minimizes the energy $I_{D;a,b}$ among those passing through $\underline z^n$. 
    By Proposition~\ref{prop:minimizer}, each $\gamma^n$ is piecewise geodesic relative to $\underline z^n$ and hence 
    \begin{equation}
        I_{D;b,a}(\gamma^n) = I_{D;a,b}(\gamma^n) \leq I_{D;a,b}(\gamma) < \infty.
    \end{equation}
    By Lemma~\ref{lem:lsc}, there exists a subsequence $\gamma^{n_k}$ which converges to a chord $\eta \in \mc X(D;b,a)$ in the Hausdorff distance, which furthermore satisfies
    \begin{equation}
        I_{D;b,a}(\eta) \leq \liminf_{k\to\infty} I_{D;b,a}(\gamma^{n_k}) = \liminf_{k\to\infty} I_{D;a,b}(\gamma^{n_k}) \leq I_{D;a,b}(\gamma).
    \end{equation}
    We claim that $\eta=\gamma$ modulo orientation. Note that for every $z \in \underline z^{m}$ and $n_k\geq m$, since $\underline z^{m} \subset \underline z^{n_k}$, the distance between $z$ and the chord $\eta$ is bounded above by the Hausdorff distance between $\gamma^{n_k}$ and $\eta$. Taking $n_k\to\infty$, we see that every $z \in \bigcup_{m\geq 1} \underline z^m$ is in $\eta$. Since $\bigcup_{m\geq 1} \underline z^m$ is dense in $\gamma$ and $\eta$ is closed, we have $\gamma \subseteq \eta$ and thus $\gamma = \eta$ by Lemma~\ref{lem:chord-inclusion}. Therefore, $I_{D;b,a}(\gamma) \leq I_{D;a,b}(\gamma)$ and, changing the roles of $a$ and $b$, we conclude $I_{D;a,b}(\gamma) \leq I_{D;b,a}(\gamma)$ as well. 

    This proves the theorem if either $I_{D;a,b}(\gamma)$ or $I_{D;b,a}(\gamma)$ is finite. It is trivially satisfied when the energy is infinite in both orientations, so $I_{D;a,b}(\gamma) = I_{D;b,a}(\gamma)$ for any $\gamma \in \mc X(D;a,b)$.
\end{proof}

\section{Energy minimizers for isotopy classes of chords}\label{sec:isotopy}

Let $D$ be a simply connected domain with distinct prime ends $a$ and $b$. Given $\underline z = \{z_1,\dots,z_n\} \subset D$, recall that $\mc X(D;a,b;\underline z)$ is the set of chords in $D$ from $a$ to $b$ which visits all points in $\underline z$. We say that two chords $\gamma_0,\gamma_1 \in \mc X(D;a,b;\underline z)$ are \textit{isotopic relative to $\underline z$} if there exists an isotopy between them that fixes $\underline z$ pointwise: i.e., a continuous map $F: [0,1]\times \gamma_0 \to D \cup \{a,b\}$ with $F(0,\gamma_0) = \gamma_0$ and $F(1,\gamma_0) = \gamma_1$ such that, for every $s\in [0,1]$, we have $F(s,\gamma_0) \in \mc X(D;a,b;\underline z)$ and $F(s,z) = z$ for each $z\in \underline z$. This is an equivalence relation; let us denote
\begin{equation}
    \mc X(D;a,b;\underline z, \gamma_0):= \{ \gamma \in \mc X(D;a,b;\underline z): \text{$\gamma$ and $\gamma_0$ are isotopic relative to $\underline z$} \}.
\end{equation}
Note that this isotopy class is preserved naturally under conformal transformations: that is, if $f:D\to f(D)$ is a conformal map, then $\gamma \in \mc X(D;a,b;\underline z, \gamma_0)$ if and only if $f(\gamma) \in \mc X(f(D);f(a),f(b);f(\underline z), f(\gamma_0))$. 

We note that an isotopy between two chords $\gamma_0$ and $\gamma_1$ relative to $\underline z$ extends to an \textit{ambient isotopy} of $\overline D$ relative to $\underline z$, where by $\overline D$ we mean the union of $D$ with the set $\partial D$ of its prime ends. The converse relationship is clear: if there exists an ambient isotopy between two chords relative to $\underline z$, then they are isotopic relative to $\underline z$. Let us denote by $\mathrm{Homeo}(\overline D,K)$ the space of self-homeomorphisms of $\overline D$ which restrict to the idenity map on $K\subset \overline D$, equipped with the compact-open topology. 

\begin{lem}\label{lem:ambient-isotopy}
    If two chords $\gamma_0,\gamma_1 \in \mc X(D;a,b;\underline z)$ are isotopic relative to $\underline z$, then there exists a continuous map $\widetilde F: [0,1] \times \overline D \to \overline D$ such that $F(s,\cdot) \in \mathrm{Homeo}(\overline D, \underline z \cup \partial D)$ for each $s\in[0,1]$, $F(0,\cdot)$ is the identity map, and $\gamma_1 = \widetilde F(1,\gamma_0)$. 
\end{lem}
\begin{proof}
    Without loss of generality, assume that $D$ is the unit disk $\ud$ and denote $\underline z = \{z_1,\dots,z_n\}$. Choose $\ep>0$ such that the closed disks $\overline{B_{2\ep}(z_k)}$ are contained in $\ud$ and disjoint from each other. It is straightforward to find a family of maps $(\phi_{\underline w})$, indexed by $\underline w = \{w_1,\dots,w_n\}$ satisfying $\max_k |w_k - z_k| < \ep$, with the following conditions: $\phi_{\underline z}$ is the identity map, $\phi_{\underline w} \in \mathrm{Homeo}(\overline \ud, \partial \ud)$ for each $\underline w$, and $\phi_{\underline w}(w_k) = z_k$ for each $w_k \in \underline w$. Furthermore, we can choose these maps so that $\underline w \mapsto \phi_{\underline w}$ is continuous with respect to the Hausdorff distance on the domain and the uniform norm on the codomain.

    Let $\gamma_0 :[0,1] \to \ud \cup \{a,b\}$ be a continuous parameterization of a chord in $\mc X(D;a,b,\underline z)$ with $t_k:= (\gamma_0)^{-1}(z_k)$ for each $z_k\in \underline z$. We claim that there exists a $\delta>0$ (depending on $\gamma_0$) so that the following is true: if $\gamma_1:[0,1] \to \ud \cup \{a,b\}$ is a parameterized chord with $\sup_{t\in [0,1]}|\gamma_0(t) - \gamma_1(t)| < \delta$ and $\gamma_1(t_k) = z_k$ for every $k$, then there exists an ambient isotopy $\widetilde F$ as in the statement of the lemma.
    By the Schoenflies theorem, we can choose $\delta>0$ such that for any such $\gamma_1$, there is $f\in \mathrm{Homeo}(\overline \ud, \partial \ud)$ with $\sup_{x\in \overline \ud} |f(x) - x| < \ep$ and $\gamma_1 = f\circ \gamma_0$ (see \cite[Cor.\ 2.9 and Thm.\ 2.11]{Pommerenke_boundary}). Since $\mathrm{Homeo}(\overline \ud, \partial \ud)$ is locally contractible \cite{hom-local-contractible}, $f$ extends to an isotopy $F: [0,1] \times \overline \ud \to \overline \ud$ from the identity map to $f$ such that $F(s,\cdot) \in \mathrm{Homeo}(\overline \ud,\partial \ud)$ and $\sup_{x\in \overline \ud}|F(s,x) - x| < \ep$ for every $s \in [0,1]$. Then, $\widetilde F(s,\cdot) := \phi_{F(s,\underline z)} \circ F(s,\cdot)$ is the desired ambient isotopy between $\gamma_0$ and $\gamma_1$ leaving $\underline z$ pointwise fixed.

    Now, given any $\gamma_1 \in \mc X(\ud;a,b;\underline z, \gamma_0)$, choose an isotopy $H: [0,1]\times \gamma_0 \to \ud \cup \{a,b\}$ from $\gamma_0$ to $\gamma_1$ leaving $\underline z$ pointwise fixed and denote $\gamma_s(t):= H(s,t)$. By the previous paragraph, there exists a $\delta_s>0$ for every $s\in [0,1]$ such that $\gamma_{\tilde s}$ is ambient isotopic to $\gamma_s$ through $\mathrm{Homeo}(\overline \ud, \underline z \cup \partial \ud)$ whenever $|\tilde s - s| <\delta_s$. This gives an open cover of $[0,1]$, so we conclude that $\gamma_1$ is ambient isotopic to $\gamma_0$ through $\mathrm{Homeo}(\overline \ud, \underline z \cup \partial \ud)$. 
\end{proof}

\begin{remark}\label{rem:braid}
All chords in $\mc X(D;a,b)$ are isotopic to each other by the Jordan curve theorem. 
When $\underline z$ contains $1\leq n <\infty$ number of points, by Lemma~\ref{lem:ambient-isotopy}, the collection of isotopy classes of chords in $\mc X(D;a,b;\underline z)$ relative to $\underline z$ correspond to the mapping class group of a disk with $n$ punctures. This can be further identified with the braid group on $n$ strands (see, e.g., \cite[Sec.\ 9.1]{mapping-class-groups}). As such, there is a unique isotopy class relative to $\underline z$ when $\underline z$ consists of a single point, but the number of isotopy classes relative to $\underline z$ is countably infinite when $\underline z$ contains more than one point.
\end{remark}

\begin{proof}[Proof of Theorem~\ref{prop:isotopy}]
    If $\underline z= \varnothing$, then $\gamma^{\mathrm{min}}$ is the hyperbolic geodesic in $\mc X(D;a,b)$ and the claim is trivially satisfied.     
    Let us thus consider the case that $\underline z$ contains $n\geq 1$ points and show that there exists a chord in $\mc X(D;a,b;\underline{z},\gamma)$ with finite energy. The proof is by induction on $n$. 
    \begin{itemize}
        \item If $\underline z$ contains a single point $z_1$, then there is a unique isotopy class of chords in $\mc X(D;a,b)$ passing through $z_1$ relative to this point. 
        In this case, $\gamma^{\mathrm{min}}$ with finite energy was uniquely identified in \cite[Prop.\ 3.1]{W1} using the interpretation of chordal Loewner energy as the large deviation rate function of chordal SLE and again in \cite[Thm.\ 3.3(i)]{energy-minimizer} using a deterministic Loewner chain argument.

        \item Let $n\geq 2$ and suppose there exists a finite energy chord passing through any set of $n-1$ points in $D$ with any given isotopy class. 
        Let $z_1$ and $z_2$ be the first and second points of $\underline z$ visited by $\gamma$, respectively. Let $\gamma_0$ be the part of $\gamma$ until it visits $z_1$, let $\gamma_1$ be the part between $z_1$ and $z_2$, and finally let $\gamma_2$ be the rest of $\gamma$ after it visits $z_2$.  

        Using the result for $n=1$ case referenced above, let us replace $\gamma_0 \cup \gamma_1$ with a finite energy chord $\eta$ in $\mc X(D \setminus \gamma_2; a,z_2)$ passing through $z_1$. Let $\eta_0$ be the part of this curve up to its visit to $z_1$ and $\eta_1$ be the remaining part from $z_1$ to $z_2$.
        Note that we can find an isotopy between $\gamma_0 \cup \gamma_1$ and $\eta$ in $D \setminus \gamma_2$ fixing $z_1$ and extend it to one between $\gamma$ and $\eta \cup \gamma_2$ in $D$ fixing $\underline z$.
        That is, $\eta \cup \gamma_2 \in \mc X(D;a,b;\underline z, \gamma)$.
        Moreover, since $I_{D\setminus \gamma_2;a,z_2}(\eta_0)<\infty$, we have $I_{D;a,b}(\eta_0)<\infty$ by \cite[Lem.\ 4.3]{W1}. 
        
        We now recall the induction hypothesis to choose a chord $\tilde \eta \in \mc X(D\setminus \eta_0;z_1,b;\underline z\setminus \{z_1\}, \eta_1 \cup \gamma_2)$ with finite energy. Then, by the additivity \eqref{eq:additivity} of Loewner energy, we have
        \begin{equation}
            I_{D;a,b}(\eta_0 \cup \tilde \eta) = I_{D;a,z_1}(\eta_0) + I_{D\setminus \eta_0;z_1,b}(\tilde \eta) < \infty.
        \end{equation}
        Furthermore, $\eta_0 \cup \tilde \eta$ is isotopic to $\eta \cup \gamma_2$ relative to $\underline z$, so we have found a finite-energy chord in $\mc X(D;a,b;\underline z, \gamma)$ as claimed.
        
    \end{itemize}

Let us assume $D = \ud$, $a=-1$, and $b=1$ and choose a sequence of finite-energy chords $(\gamma^m)_{m\in \mb N}$ in $\mc X(\ud;-1,1;\underline z,\gamma)$ whose energies decrease to the infimum energy in this isotopy class. By Lemma~\ref{lem:lsc}, for each $m$, there exists $\varphi_m \in \mathrm{Homeo}(\overline \ud, \partial \ud)$ with $\varphi_m([-1,1]) = \gamma^m$. Moreover, along a subsequence, $\varphi_{m_k}$ converges uniformly to some $\varphi \in \mathrm{Homeo}(\overline \ud, \partial \ud)$. Since $\overline \ud$ is compact, $\varphi_{n_k}^{-1} \to \varphi^{-1}$ uniformly as well, which implies that $\varphi_{m_k}^{-1}(\underline z)$ converges in the Hausdorff distance. Then, we can pick a sequence of homeomorphisms $\psi_{m_k} \in \mathrm{Homeo}(\overline \ud, \partial \ud)$ converging uniformly to the identity map such that $\psi_{m_k}([-1,1]) = [-1,1]$ and $(\varphi_{m_k} \circ \psi_{m_k})^{-1} (\underline z) = \varphi^{-1}(\underline z)$ for every $m_k$. Letting $\gamma^{\mathrm{min}}:= \varphi([-1,1])$, we see that 
$\varphi_{m_k} \circ \psi_{m_k} \circ \varphi^{-1} \in \mathrm{Homeo}(\overline \ud, \underline z \cup \partial \ud)$ maps $\gamma^{\mathrm{min}}$ to $\gamma^{m_k}$, and these homeomorphisms converge uniformly to the identity map as $m_k\to\infty$. Since the identity component of $\mathrm{Homeo}(\overline \ud, \underline z \cup \partial \ud)$ is homotopically trivial\cite{hamstrom-disk-holes}, we conclude $\gamma^{\mathrm{min}} \in \mc X(\ud;-1,1;\underline z,\gamma)$. Since $\gamma^{m_k} \to \gamma^{\mathrm{min}}$ in the Hausdorff distance, we obtain \eqref{eq:isotopy-min} by the lower semicontinity of Loewner energy (Lemma~\ref{lem:lsc}).

The claim that an energy minimizer of $\mc X(D;a,b;\underline z,\gamma)$ is piecewise geodesic relative to $\underline z$ follows from the fact that the manipulations in the proof of Proposition~\ref{prop:minimizer} do not change the isotopy class.
\end{proof}

\bibliographystyle{alpha}
{\small\bibliography{ref}}

\end{document}